\documentclass[letter]{amsart}
\usepackage{amssymb}
\usepackage[bookmarks, bookmarksdepth=2, colorlinks=true, linkcolor=blue, citecolor=blue, urlcolor=blue]{hyperref}
\usepackage{color}
\usepackage{cleveref}
\usepackage{algorithm}
\newtheorem{theorem}{\bf Theorem}[section]
\newtheorem{proposition}[theorem]{\bf Proposition}

\theoremstyle{definition}
\newtheorem{definition}{\bf Definition}

\newcommand \Z{{\mathbb Z}}
\newcommand \Q{{\mathbb Q}}
\newcommand \R{{\mathbb R}}
\newcommand \C{{\mathbb C}}

\DeclareMathOperator{\rank}{rank}

\newcommand \x{{\mathbf{x}}}

\newcommand{\arxiv}[1]{\href{http://arxiv.org/abs/#1}{{\tt arXiv:#1}}}

\begin{document}
\title{\bf{Computing symmetric determinantal representations}}

\author{Justin Chen}
\address{School of Mathematics, Georgia Institute of Technology,
Atlanta, Georgia, 30332 U.S.A.}
\email{jchen646@math.gatech.edu}

\author{Papri Dey}
\address{Simons Institute for the Theory of Computing, Berkeley, California, 94720 U.S.A.}
\email{papridey@berkeley.edu}



\begin{abstract}
We introduce the \texttt{DeterminantalRepresentations} package for \textit{Macaulay2},
which computes definite symmetric determinantal representations of real
polynomials. We focus on quadrics and plane curves of low degree (i.e. cubics
and quartics). Our algorithms are geared towards speed and robustness, employing 
linear algebra and numerical algebraic geometry, without genericity assumptions on
the polynomials.
\end{abstract}

\maketitle

\section{Introduction}
The problem of representing a polynomial as the determinant of a linear matrix pencil is classical, cf. \cite{B,BK,Dic,D,HV}. A polynomial $f \in \R[x_1, \ldots, x_n]$ of degree $d$ (not necessarily homogeneous) is called \textit{determinantal} if $f$ is the determinant of a matrix with linear entries -- i.e. there exist matrices $A_0, \ldots, A_n \in \R^{d\times d}$ such that $f(x_1, \ldots, x_n) = \det(A_0 + x_1A_1 + \ldots + x_nA_n)$. The matrix $A_0 + x_1A_1 + \ldots + x_nA_n$ is said to give a \textit{determinantal representation} of $f$ of size $d$. If the matrices $A_i$ can be chosen to be all symmetric (resp. Hermitian), then the determinantal representation is called \textit{symmetric} (resp. Hermitian). The determinantal representation is called \textit{definite} if $A_0$ is definite, and \textit{monic} if $A_0 = I_d$ is the identity matrix. 

Computing definite symmetric (resp. Hermitian) determinantal representations of a polynomial is known as the determinantal representation problem in convex algebraic geometry \cite{BPR}. It has generated interest to the optimization community due to its connection with the problem of determining (definite) linear matrix inequality (LMI) representable sets (\cite{HV,V}). The problem of characterizing the LMI-representable subsets of $\R^n$ (i.e. spectrahedra) can be answered by characterizing determinantal polynomials, which leads to the generalized Lax conjecture \cite{LPR}.

Throughout we focus mainly on homogeneous polynomials, typically in $3$ variables, corresponding to projective plane curves (though internally via de-homogenization, it suffices to compute determinantal representations for bivariate polynomials). By a celebrated theorem of Helton-Vinnikov (\cite{HV, LPR}), all hyperbolic polynomials in $3$ variables admit symmetric determinantal representations. When $n \ge 4$, a general homogeneous polynomial of degree $d$ in $n$ variables does not admit any determinantal representation of size $d$ (except for $(n, d) = (4, 3)$). We abbreviate monic symmetric (resp. Hermitian) determinantal representation to MSDR (resp. MHDR).

\section{Quadratic determinantal polynomials}
For a quadratic polynomial $f(\x)=\x^{T}A\x+b^{T}\x+1 \in \R[\x]$ in $n$ (= any number of) variables, one can give a necessary and sufficient condition for a MSDR (resp. MHDR) to exist, via an associated matrix $W := A-\frac{1}{4}bb^{T}$; namely one (or both) of the following conditions must hold:

\begin{enumerate}
\item $W$ is negative semidefinite with $\rank W \leq 3$
\item $A$ is negative semidefinite.
\end{enumerate}

\noindent
One can explicitly find a MSDR (resp. MHDR) of size $2$ in case (1), and one of size at most $n+1$ in case (2), cf. (\cite{papriquadratic}, Theorem 3.4) for an algorithm. This is implemented as follows:

\vspace{0.05cm}

\begin{verbatim}
Macaulay2, version 1.13
i1 : needsPackage "DeterminantalRepresentations"
i2 : R = CC[x_1..x_4]
i3 : f = -25*x_1^2 + 254*x_1*x_2 + 243*x_2^2 + 234*x_1*x_3 + 494*x_2*x_3 + 247*x_3^2 + 
     198*x_1*x_4 + 378*x_2*x_4 + 378*x_3*x_4 + 143*x_4^2 + 18*x_1 + 32*x_2 + 32*x_3 + 
     24*x_4 + 1
i4 : quadraticDetRep f
o4 = | 8.06514x_1+19.0421x_2+17.2128x_3+12.4043x_4+1 10.2531x_1+1.93541x_2+2.74393x_3+... |
     | 10.2531x_1+1.93541x_2+2.74393x_3+.914643x_4   9.93486x_1+12.9579x_2+14.7872x_3+... |
\end{verbatim}

\section{Generalized mixed discriminant}
The relations between the coefficients of a determinantal polynomial and the entries of the coefficient matrices are captured by the generalized mixed discriminant:

\begin{definition} \label {generalizedmixdef}
Let $\{A^{(1)}, \ldots, A^{(n)}\} = \{(a_{ij}^{(1)}), \ldots, (a_{ij}^{(n)})\}$ be a set of $n \times n$ matrices of size $n$. The generalized mixed discriminant of a tuple of matrices $(\underbrace{A^{(1)}, \dots , A^{(1)}}_{k_{1}}, \underbrace{A^{(2)}, \dots, A^{(2)}}_{k_{2}},\dots,\underbrace{A^{(n)},\dots,A^{(n)}}_{k_{n}})$ is defined as
\begin{equation} \nonumber
\widehat{D}(\underbrace{A^{(1)}, \dots , A^{(1)}}_{k_{1}}, \underbrace{A^{(2)}, \dots, A^{(2)}}_{k_{2}},\dots,\underbrace{A^{(n)},\dots,A^{(n)}}_{k_{n}}) :=
\sum_{\alpha \in S[m]} \sum_{\sigma \in \widetilde{S}}
   \det \begin{bmatrix} a_{\alpha_{1}\alpha_{1}}^{(\sigma(1))} & \dots & a_{\alpha_{1}\alpha_{k}}^{(\sigma(1))} \\
     \vdots & & \\
     a_{\alpha_{k}\alpha_{1}}^{(\sigma(k))} & \dots & a_{\alpha_{k}\alpha_{k}}^{(\sigma(k))} \end{bmatrix}
\end{equation} where $m$ is the number of distinct matrices, $S[m]$ is the set of order-preserving $m$-cycles in $S_n$ (i.e. $\alpha = (\alpha_{1}, \dots, \alpha_{m}) \in S[m] \Rightarrow \alpha_{1} < \alpha_{2}< \dots < \alpha_{m}$), and 
$\widetilde{S}$ is the set of all distinct permutations of $\{\underbrace{1, \dots, 1}_{k_{1}},\dots,\underbrace{n, \dots, n}_{k_{n}}\}$.
\end{definition}

\begin{theorem} \label{themgmd} (cf. \cite{paprimulti}, Theorem 2.5)
If $f = \det(I_d + \sum x_iA_i)$ is determinantal of degree $d$, then the coefficients of $f$ are given by the generalized mixed discriminants of the matrices $A_{i}$: namely the coefficient of $x_{1}^{k_{1}} \dots x_{n}^{k_{n}}$ in $f$ is equal to $\widehat{D}(\underbrace{A_{1}, \dots , A_{1}}_{k_{1}}, \underbrace{A_{2}, \dots, A_{2}}_{k_{2}},\dots,\underbrace{A_{n},\dots,A_{n}}_{k_{n}}).$
\end{theorem}

\noindent
Generalized mixed discriminants can be computed with the function {\tt generalizedMixedDiscriminant}:

\begin{verbatim}
i5 : n = 3; S = QQ[a_(1,1)..a_(n,n),b_(1,1)..b_(n,n),c_(1,1)..c_(n,n)]; R = S[x_1..x_n]
i8 : A = sub(genericMatrix(S, n, n), R)
i9 : B = sub(genericMatrix(S, b_(1,1), n, n), R)
i10 : C = sub(genericMatrix(S, c_(1,1), n, n), R)
i11 : G = generalizedMixedDiscriminant({A, B, C})
o11 = a   b   c    - a   b   c    - a   b   c    + a   b   c    - a   b   c    + ...
       3,3 2,2 1,1    3,2 2,3 1,1    2,3 3,2 1,1    2,2 3,3 1,1    3,3 2,1 1,2
i12 : P = det(id_(R^n) + x_1*A + x_2*B + x_3*C);
i13 : G == (last coefficients(P, Monomials => {x_1*x_2*x_3}))_(0,0)
o13 = true
\end{verbatim}

\section{Higher degree determinantal polynomials}
We now consider polynomials of degree $> 2$. Let $f \in \R[x_0, x_1, x_2]_d$ be a homogeneous polynomial of degree $d$ in $3$ variables. The homogenization of a MSDR for $f \Big|_{x_0=1}$ gives a MSDR for $f$, so we reduce to the case of $2$ variables. Given a (not necessarily homogeneous) polynomial $f \in \R[x_1, x_2]$ of total degree $d$, we seek to compute $A_1, A_2 \in \operatorname{Sym}_{d \times d}(\R)$ such that $f = \det(I_d + x_1 A_1 + x_2 A_2)$. It is easy to obtain the eigenvalues of the unknown matrices $A_i$: for instance, the eigenvalues of $A_1$ are the negative reciprocals of the roots of the univariate polynomial $f \Big|_{x_2 = 0} = \det(I_d + x_1A_1)$, and similarly for $A_2$ (note that this polynomial has nonzero roots). 

Now, if the $A_i$ are symmetric, then by the spectral theorem there are orthogonal matrices $V_i$ such that $V_i^TA_iV_i = D_i$ is diagonal, with entries equal to the known eigenvalues of $A_i$, for $i = 1, 2$. Setting $V := V_2^TV_1$ (which is orthogonal), one has $f = \det(I_d + x_1 D_1 + x_2V^TD_2V) = \det(I_d + x_1 VD_1V^T + x_2D_2)$. 

With this, one can obtain the diagonal entries of $V^TD_2V$: it follows from \Cref{themgmd} that these can be obtained by solving a linear system involving only $D_1$ (i.e. eigenvalues of $A_1$ -- note the subscript) as well as coefficients of monomials in $f$ which are linear in $x_2$. We note that the linear system giving rise to diagonal entries of $V^TD_2V$ has unique solutions iff $A_1$ has distinct eigenvalues -- if the solutions are nonunique, then (as we will see) we must choose a solution which is majorized by (the diagonal entries of) $D_2$. By a symmetrical argument, we may henceforth assume that diagonal entries of $V^TD_2V$ and $VD_1V^T$ are known.

It thus suffices to compute the symmetric matrix $V^TD_2V$: we propose two methods to do so. The first, brute-force, method is to simply find the $d \choose 2$ off-diagonal entries of $V^TD_2V$ by solving a square polynomial system arising from \Cref{themgmd}. Although symbolic methods are slow, a numerical method for solving the system is implemented in the method \texttt{trivariateDetRep}, with the (default) option \texttt{Strategy => DirectSystem}:

\begin{verbatim}
i14 : R = RR[x,y,z]
i15 : f = det(x*id_(R^4) + y*diagonalMatrix {4,3,2,1_R} + z*randomIntegerSymmetric(4,R))
i16 : sols = trivariateDetRep f;
Solving 6 x 6 polynomial system ...
i17 : #sols
o17 = 64
i18 : all(sols, M -> clean(1e-7, f - det M) == 0)
o18 = true
\end{verbatim}

\noindent
This method works well up to degree $4$, but for higher degrees, even numerical methods take too long to finish.

An alternative, more theoretical, method for cubics is to note that since $D_2$ is known, finding $V^TD_2V$ is equivalent to finding $V$, and in this case (of degree $3$), it turns out that the Hadamard square of $V$ can be determined essentially by linear algebra
(cf. \cite{papribiv}, Theorem 2.14). 


\begin{definition}
If $A, B$ are matrices of the same size, we denote their \textit{Hadamard product} by $A \odot B$, i.e. $(A \odot B)_{ij} = A_{ij}B_{ij}$. We say that a square matrix $A$ is \textit{orthostochastic} if it is the Hadamard square of an orthogonal matrix, i.e. $A = V \odot V$ for some orthogonal matrix $V$.
If $v, w \in \R^n$, we say that $v$ is \textit{majorized by} $w$ if $\sum_{j=1}^n v_j = \sum_{j=1}^n w_j$ and $\sum_{j=1}^i (\widetilde{v})_j \le \sum_{j=1}^i (\widetilde{w})_j$ for all $i = 1, \ldots, n$, where $\widetilde{v}, \widetilde{w}$ are the decreasing rearrangements of $v, w$.
\end{definition}

To elaborate: in the cubic case, given the majorization conditions mentioned above, one can set up a zero-dimensional polynomial system to find the unknown entries of $V \odot V$, which consists of one cubic equation with all other equations linear, and only involves the (known) diagonal entries of $D_1, D_2, V^TD_2V, VD_1V^T$. 

Thus to recover $V$, it suffices to determine all orthogonal matrices with Hadamard square equal to a given orthostochastic matrix. Given a $n\times n$ orthostochastic matrix $A$, there are $2^{n^2}$ possible matrices whose Hadamard square is $A$ (not all of which will be orthogonal in general though). Let $G\cong (\Z/2\Z)^n$ be the group of diagonal matrices with diagonal entries equal to $\pm 1$. Then $G \times G$ acts on the set of orthogonal matrices whose Hadamard square is $A$, via $(g_1, g_2) \cdot O := g_1Og_2$. The method {\tt orthogonalFromOrthostochastic} computes all such orthogonal matrices, modulo the action of $G \times G$. We note the following:

\begin{proposition}
For a general orthostochastic matrix $A$, there is exactly one $G \times G$-orbit of orthogonal matrices with Hadamard square equal to $A$.
\end{proposition}

\begin{proof}
Since the action of $G \times G$ amounts to performing sign changes in each row and column, each $G \times G$ orbit contains a unique element with nonnegative entries in the first row and first column. If there were distinct $G \times G$ orbits for $A$, then there would be orthogonal matrices $U, V$ with identical first column $u_{(i,1)} = v_{(i,1)} = \sqrt{a_{(i,1)}}$, but $u_j \ne \pm v_j$ for some column $j > 1$ (since $u_{(1,j)} = v_{(1,j)} = \sqrt{a_{(1,j)}} \ge 0$), which would impose a Zariski-closed condition on the entries of $U$ (namely $u_1^T(u_j - v_j) = 0$), hence does not hold for general $A$. 
\end{proof}

We illustrate this with some examples: note that when using floating point inputs, it may be necessary to specify the value of the option {\tt Tolerance} (default value $10^{-5}$) in order to obtain useful results.

\begin{verbatim}
i19 : (A1, A2) = (randomIntegerSymmetric(3, R), randomIntegerSymmetric(3, R))
o19 = (| 12 3  18 |, | 18 16 9  |)
       | 3  12 14 |  | 16 4  11 |
       | 18 14 6  |  | 9  11 16 |
i20 : f = det(x*id_(R^3) + y*A1 + z*A2)
       3      2          2        3      2                  2         2          2        3
o20 = x  + 30x y - 241x*y  - 3918y  + 38x z + 52x*y*z + 768y z - 34x*z  + 3282y*z  - 2278z
i21 : reps = trivariateDetRep f
o21 = {| x+33.7014y+36.8578z .607983z            3.51039z            |,
       | .607983z            x+9.08918y-6.37807z -3.92965z           | 
       | 3.51039z            -3.92965z           x-12.7906y+7.52028z | 
      ------------------------------------------------------------------
      | x+33.7014y+36.8578z 1.11917z            3.2642z             |}
      | 1.11917z            x+9.08918y-6.37807z 4.02828z            |
      | 3.2642z             4.02828z            x-12.7906y+7.52028z |
i22 : all(reps, M -> clean(1e-9, f - det M) == 0)
o22 = true
i23 : g = x^3+7*x^2*y+16*x*y^2+12*y^3+3*x^2*z+22*x*y*z+32*y^2*z-45*x*z^2-65*y*z^2-175*z^3
o24 : reps = trivariateDetRep g
o24 = {| x+3y-7z 0       0       |}
       | 0       x+2y+5z 0       |
       | 0       0       x+2y+5z |

\end{verbatim}

\noindent
As seen above, the output of {\tt trivariateDetRep} is a list of matrices $M$, whose entries are linear forms, such that the input polynomial $f$ is equal to $\det M$. As the second example shows, no genericity assumptions are made on the polynomial $f$ -- that is, the eigenvalues of the coefficient matrices $A_1$ and $A_2$ need not be distinct (as opposed to the treatment in e.g. \cite{PSV}). This in turn reveals information which is typically hidden to numerical methods (such as a \texttt{numericalIrreducibleDecomposition}), e.g. that one of the component lines of the plane curve defined by $g$ above has multiplicity $2$.

\section{Additional methods}

In the course of creating this package, various functions for working with matrices were needed which (to the best of our knowledge) were not available in \textit{Macaulay2}. Thus a number of helper functions are included in this package, which may be of general interest to users beyond the scope of computing determinantal representations. These include: \texttt{hadamard} (for computing Hadamard products), \texttt{cholesky} (for computing the Cholesky decomposition of a PSD matrix), \texttt{companionMatrix} (which returns a matrix whose characteristic polynomial is any given univariate monic polynomial), \texttt{isOrthogonal, isDoublyStochastic} (for checking properties of a given matrix), \texttt{randomIntegerSymmetric, randomUnipotent, randomOrthogonal, randomPSD} (for generating various types of random matrices), and \texttt{liftRealMatrix, roundMatrix} (for lifting matrices from $\C$ to $\R$ and $\Q$).

\vspace{0.5cm}

\noindent \textsc{Acknowledgements.} Both authors gratefully acknowledge the support of ICERM, and the Fall 2018 Nonlinear Algebra program in particular, where this project began. We thank Ritvik Ramkumar for helpful comments and software testing. The second author would like to thank her Ph.D. thesis supervisor Harish K. Pillai for his helpful suggestions in this topic and Deepak Patil for implementing the notion of generalized mixed discriminant of matrices in Matlab during her Ph.D.

\end{document}